\documentclass[a4paper,11pt]{amsart}
\title[A new concordance homomorphism from Khovanov homology]{A new concordance homomorphism from integral Khovanov homology}
\author{Lukas Lewark}
\address{ETH Z\"urich, R\"amistrasse 101, 8092 Z\"urich, Switzerland}
\email{\myemail{llewark@math.ethz.ch}}
\urladdr{\url{https://people.math.ethz.ch/~llewark/}}
\keywords{Knots, Concordance, Slice genus, Khovanov homology, Rasmussen invariants}
\subjclass{57K10, 57K18}
\usepackage[latin1]{inputenc}
\usepackage{amsmath,amssymb,stmaryrd,thmtools,enumerate,graphicx,xcolor,enumitem,afterpage,float,amscd,url,mathtools,microtype,tikz-cd,caption,amsthm,calc,nicematrix,multicol}
\makeatletter
\def\@tocline#1#2#3#4#5#6#7{\relax
  \ifnum #1>\c@tocdepth 
  \else
    \par \addpenalty\@secpenalty\addvspace{#2}%
    \begingroup \hyphenpenalty\@M
    \@ifempty{#4}{%
      \@tempdima\csname r@tocindent\number#1\endcsname\relax
    }{%
      \@tempdima#4\relax
    }%
    \hspace{-2.1\parindent} \leftskip#3\relax \advance\leftskip\@tempdima\relax
    \rightskip\@pnumwidth plus4em \parfillskip-\@pnumwidth
    #5\leavevmode\hskip-\@tempdima #6\nobreak\relax
    \ifnum#1<0\hfill\else\dotfill\fi\hbox to\@pnumwidth{\@tocpagenum{#7}}\par
    \nobreak
    \endgroup
  \fi}
\makeatother
\newcommand{\nocontentsline}[3]{}
\newcommand{\tocless}[2]{\bgroup\let\addcontentsline=\nocontentsline#1{#2}\egroup}
\usetikzlibrary{patterns,decorations.pathreplacing}
\usepackage[bookmarks=true,pdfencoding=auto,hidelinks]{hyperref}
\hypersetup{pdfauthor={\authors},pdftitle={\shorttitle}}
\usepackage[nameinlink]{cleveref}
\let\cref\Cref
\crefname{subsection}{section}{sections}
\Crefname{subsection}{Section}{Sections}
\Crefname{enumi}{}{}
\crefname{equation}{}{}
\definecolor{darkblue}{RGB}{0,0,96}
\definecolor{gray}{RGB}{127,127,127}
\definecolor{darkred}{RGB}{160,0,0}
\definecolor{lightyellow}{RGB}{255,255,128}
\newcommand{\myemail}[1]{\href{mailto:#1}{#1}}
\newcommand*{\norangehyphen}{-}
\newcommand{\qua}{\hskip 0.4em \ignorespaces}
\def\arxiv#1{\relax\ifhmode\unskip\qua\fi
\href{http://arxiv.org/abs/#1}%
{\tt arXiv:\penalty -100\unskip#1}}

\def\MR#1{\relax\ifhmode\unskip\qua\fi
\href{https://mathscinet.ams.org/mathscinet-getitem?mr=#1}{\tt MR#1}}
\def\ZB#1{\relax\ifhmode\unskip\qua\fi
\href{https://zbmath.org/?q=an:#1}{\tt Zbl\:#1}}
\def\xox#1{\csname xx#1\endcsname}

\renewenvironment{thebibliography}[1]{
  \begin{oldthebibliography}{#1}\small
    \setlength{\itemsep}{.5ex}
    \setlength{\parskip}{0em}
}
{
  \end{oldthebibliography}
}
\let\stdthebibliography\thebibliography
\let\stdendthebibliography\endthebibliography
\renewenvironment*{thebibliography}[1]{%
  \stdthebibliography{KWZ19.}}
  {\stdendthebibliography}

\numberwithin{equation}{section}
\declaretheorem[numberwithin=section]{lemma}
\newtheorem{theorem}[lemma]{Theorem}

\newtheorem{proposition}[lemma]{Proposition}
\newtheorem{conjecture}[lemma]{Conjecture}
\newtheorem*{theorem*}{Theorem}
\newtheorem{question}[lemma]{Question}
\newtheorem*{question*}{Question}
\theoremstyle{definition}
\newtheorem{definition}[lemma]{Definition}

\newtheorem{remark}[lemma]{Remark}

\newtheorem{example}[lemma]{Example}
\DeclareMathAlphabet{\mathpzc}{OT1}{pzc}{m}{it}

\newcommand{\Z}{\mathbb{Z}}

\newcommand{\Hom}{\text{Hom}}

\newcommand{\id}{\text{id}}

\newcommand{\qdeg}{\operatorname{qdeg}}
\newcommand{\qrk}{\operatorname{qrk}}
\newcommand{\tqrk}{\operatorname{tqrk}}
\newcommand{\tdeg}{\operatorname{tdeg}}
\newcommand{\tqdeg}{\operatorname{tqdeg}}
\newcommand{\nbd}{\nobreakdash--}

\graphicspath{{figures/}}
\DeclareMathOperator{\gl}{gl}
\DeclareMathOperator{\Char}{char}
\definecolor{darkgreen}{rgb}{0.0, 0.7, 0.0}
\newcommand{\stair}{\Sigma}
\newcommand{\calstair}{\mathcal{S}}
\begin{document}
\begin{abstract}
The universal Khovanov chain complex of a knot
modulo an appropriate equivalence relation
is shown to yield a homomorphism on the
smooth concordance group, which is strictly
stronger than all Rasmussen invariants over fields of different characteristics combined.
\end{abstract}
\maketitle
\section{Introduction}
Khovanov homology's \cite{kh1}
first geometric application
was given by the Rasmussen invariant~\cite{ras3}: a homomorphism $s_0$ from the smooth knot
concordance group $\mathcal{C}$ to $\Z$, providing
a lower bound for the smooth slice genus strong enough to
determine the smooth slice genus for torus knots.
Although $s_0$ is modeled on the concordance homomorphism $\tau$ from Heegaard--Floer homology~\cite{osz10,phdrasmussen}, and similar to it in many respects, there 
are also applications of $s_0$ that are not available using~$\tau$~\cite{zbMATH07168645}.

The Rasmussen invariant $s_0$ is extracted from rational Khovanov homology.
Similarly, homologies with $\mathbb{F}_p$ coefficients for prime $p$
yield further Rasmussen invariants~$s_p$~\cite{mvt},
and it is conjectured that $s_0, s_2, s_3, \ldots$ are linearly independent~\cite{zbMATH06296598,postcard,schuetz,LZ22}.
But what further concordance information is there in Khovanov homology?
One fruitful approach is to extract such information from additional structure on Khovanov homology,
e.g.~from a Khovanov stable homotopy type~\cite{zbMATH06296598},
or a spectral sequence coming from triply graded homology~\cite{arXiv:2004.10807}.
The focus in this paper, however, shall be on concordance information
that, just like the Rasmussen invariants, comes from
Khovanov homology over different Frobenius algebras~\cite{zbMATH05039154},
or equivalently from Bar-Natan's homology~\cite{bncob}.

That is to say, we are dealing with the universal Khovanov chain complex~$C(D)$ associated with a diagram~$D$ of a knot~$K$, which consists of graded modules over the polynomial ring~$\Z[G]$ (where $G$ is just a formal variable).
For the experts, let us note that $C(D)$ is the reduced chain complex coming from the Frobenius algebra $\Z[G,X]/(X^2+GX)$ over the ring $\Z[G]$.
But let us emphasize that to follow this paper in detail, a modicum of homological algebra suffices and no previous knowledge of Khovanov homology is required.
Indeed, for the construction of our new concordance homomorphism, we will exclusively rely on seven axioms for Khovanov homology (listed in \cref{sec:axioms}).

The Rasmussen invariant $s_c(K)$, for $c$ equal to zero or a prime,
may be extracted from $C(D) \otimes_{\Z[G]} \mathbb{F}[G]$ for $\mathbb{F}$ a field of characteristic~$c$.
If one wishes to find concordance information beyond the $s_c(K)$ in $C(D)$,
one may try to abandon the tensoring with the coefficient module $\mathbb{F}[G]$, and instead directly work with $C(D)$.
This strategy has been employed by Lipshitz--Sarkar~\cite[Section~6]{zbMATH06296598}
to construct concordance invariants $s^{\Z,m}_{\min}(K), s^{\Z,m}_{\max}(K) \in \Z$ for $m\in \Z$
(not related to the Khovanov homotopy type constructed in the same paper);
by Sano--Sato~\cite{arXiv:2211.02494}, who construct invariants $ss_c$ and $\widetilde{ss}_c$
depending on a choice of PID $R$ and $c\in R$; and 
by Sch\"utz~\cite{schuetz}, who extracts from $C(D)$ a concordance invariant $s^{\Z}(K)$,
which takes values in the set of finite tuples of integers.
Unlike the Rasmussen invariants $s_c$, none of the above-mentioned concordance invariants appear to give rise to a new homomorphism on the concordance group.
This does not come as a complete surprise: the coefficient ring $\Z[G]$ is not a PID,
and $\mathbb{F}[G]$ being a PID is a crucial ingredient in the proof that $s_c$ is a concordance homomorphism, i.e.\ well-behaved with respect to mirrors and sums of knots.

In this paper, we present a new concordance homomorphism coming from~$C(D)$.
It is constructed by considering the set of so-called \emph{knot-like} chain complexes
(see \cref{def:knotlike}) modulo \emph{$Z$\nbd equivalence} (see \cref{def:Zeq}).
Similar constructions exist for the knot Floer complex $CFK^{\infty}$~\cite{MR3260841}
and for the Khovanov--Rozansky $\mathfrak{sl}_n$--complex~\cite{lewarklobb}.
We obtain the following.
\begin{theorem}\label{thm:construction}
Modulo $Z$\nbd equivalence, knot-like chain complexes form an abelian group $\mathcal{G}$ with group
operation the tensor product, the inverse of a complex given by its dual, and the complex
with graded rank 1 as neutral element.
Associating to a knot $K$ the $Z$\nbd equivalence class of
the universal Khovanov chain complex $C(D)$ of one its diagrams $D$
induces a group homomorphism $S$ from the smooth concordance group to $\mathcal{G}$.
\end{theorem}

\begin{theorem}\label{thm:stronger}
For all knots $K$, the Rasmussen invariants $s_c(K)$ for $c$ zero or a prime,
as well as Sch\"utz's invariant $s^{\Z}(K)$, are determined by $S(K)$.
\end{theorem}

\begin{figure}[t]%
\includegraphics{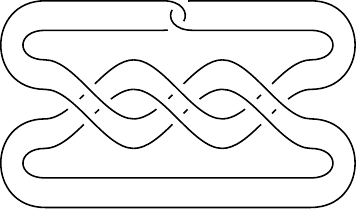}\\[5mm]
\includegraphics{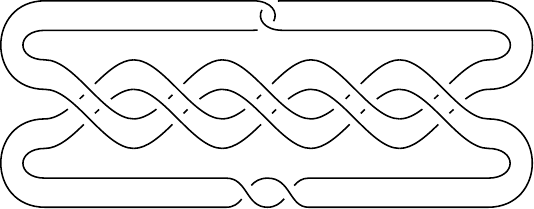}
\caption{The first two knots in \cref{conj:whitehead}:
$W^+_3(T_{2,3})$ and $W^+_6(T_{2,5})$.}
\end{figure}
To demonstrate that $S$ is stronger than all the Rasmussen invariants $s_c$ combined,
we turn to Whitehead doubles,
which seem to be particularly suitable to distinguish knot homology concordance homomorphisms~\cite{heddord,postcard,schuetz,LZ22}.
Let us denote by $W^+_m(K)$ the positively clasped $m$\nbd twisted Whitehead double of the knot $K$.
The following conjecture is inspired by and in accordance with~\cite[Conjecture~6.1]{schuetz},
and has been checked by computer calculation using \cite{khoca,homca} for $1\leq n\leq 12$~\cite{szgit}.
\begin{conjecture}\label{conj:whitehead}
For $n \geq 1$, $S(W^+_{3n}(T_{2,2n+1}))$ is represented by
\[
\stair_{(2^n)} \coloneqq
\begin{tikzcd}
q^{2}t^0 \Z[G] \ar[r,"2^n"]\ar[d,phantom,"\oplus"] & q^2t^1 \Z[G]. \\
q^{0}t^0 \Z[G] \ar[ru,"G"]
\end{tikzcd}
\]
\end{conjecture}
If this conjecture holds, then on the one hand for all $n\geq 1$ one finds $s_2(W^+_{3n}(T_{2,2n+1})) = 2$
and $s_c(W^+_{3n}(T_{2,2n+1})) = 0$ for any $c$ equal to zero or an odd prime
(compare~\cite[Conjecture~6.9]{LZ22}); but on the other hand,
then $S$ can be used to show linear independence of these knots in the smooth concordance group $\mathcal{C}$:
\begin{theorem}\label{thm:kernel}\mbox{}

\begin{enumerate}[label=(\roman*)]
\item The complexes $\stair_{(2^n)}$ for $n\geq 1$ form an infinite linearly independent family in~$\mathcal{G}$.
\item If \cref{conj:whitehead} holds,
the knots $W^+_{3n}(T_{2,2n+1})$ for $n\geq 1$
are sent by $S$ to linearly independent elements of~$\mathcal{G}$,
and hence these knots are the basis of a $\Z^{\oplus\infty}$ subgroup of~$\mathcal{C}$.
\item There exist knots $J_n$ with trivial $S(J_n)$ such that
the knots $K_n\coloneqq W^+_{3n}(T_{2,2n+1}) \# J_n$ are topologically slice.
Hence, under the assumption that \cref{conj:whitehead} holds,
the $K_n$ are a basis of a $\Z^{\oplus\infty}$ subgroup contained in the subgroup of topologically slice knots in $\mathcal{C}$.
\end{enumerate}
\end{theorem}
While this paper was being written,
Dunfield--Lipshitz--Sch\"utz
independently constructed a similar smooth concordance homomorphism,
which moreover takes odd Khovanov homology into account~\cite{2312.09114}.

\subsection*{Structure of the paper}%
\setcounter{tocdepth}{1}%
\noindent\begin{multicols}{2}%
\footnotesize%
\noindent\tableofcontents%
\end{multicols}%
\cref{sec:prelim} goes over preliminaries regarding graded rings, modules and chain complexes. Experts may skip this section or just glance at it to know which conventions are used in this paper.
\cref{sec:axioms} lists seven axioms for Khovanov homology, which are sufficient to construct (though not compute) our new homomorphism.
\cref{sec:glimpse}, which may be skipped by novices, explains what version of Khovanov homology we use and why it satisfies the axioms.
In \cref{sec:construction}, the Rasmussen invariants of fields are constructed from the axioms. This is a warm-up section without new results.
The construction of our new homomorphism $S$ and the proof of \cref{thm:construction} are contained in \cref{sec:rasint}.
The subsequent \cref{sec:calc} is concerned with calculations of $S$ and the proof of \cref{thm:kernel}(i) and (ii).
The final \cref{sec:further}
explores the relationship between $S$ and Sch\"utz's invariant $s^{\Z}$,
a bound for the slice genus given by $S$, and $S$ being a squeezedness obstructions, which leads to the proof of \cref{thm:kernel}(iii).

\subsection*{Acknowledgments}
The author warmly thanks Claudius Zibrowius and Daniel Sch\"appi for fruitful discussions.
As is often the case~\cite{kleist}, this project advanced most when talks about it were given.
So the author thanks Giuseppe Bargagnati, Filippo Bianchi, Giovanni Italiano and Alice Merz
for inviting him to give a mini-course on Khovanov homology at 
the Baby geometri seminar in Pisa, and 
Andrew Lobb, Maggie Miller, and Arunima Ray for inviting him to
the Oberwolfach Workshop \emph{Morphisms in Low Dimensions}~\cite{zbMATH07753245}
(the extended abstract of the author's talk \emph{The joy of not being a PID} at that workshop contains a summary of \cref{thm:construction,thm:stronger}).
Finally, the author gratefully acknowledges the support by the Emmy Noether Programme of the DFG, project no. 412851057,
and thanks Peter Feller for the invitation to the FIM at ETH Z\"urich during April and May 2023.
Further thanks go to Claudius Zibrowius and Laura Marino for comments on a first version of this paper.

\section{Preliminaries on graded rings}\label{sec:prelim}

For an abelian group $A$~(which in this paper will always be $\Z$ or $\Z^2$),
an $A$\nbd\emph{grading} on a ring $R$ (all our rings are commutative and unital)
is a decomposition of the additive abelian group of $R$
as~$\bigoplus_{a\in A} R_a$, such that~$R_aR_b \subset R_{a+b}$.
For a graded ring~$R$, a \emph{grading} on an $R$\nbd module~$M$ is a decomposition of the additive
abelian group of $M$ as~$\bigoplus_{a\in A} M_a$, such that~$R_aM_b \subset M_{a+b}$.
Elements of the $M_i$ are called \emph{homogeneous of grading}~$a$.
Homomorphisms $f\colon M\to N$ of graded $R$\nbd modules are called \emph{homogeneous of degree~$a$}
if $f(M_b) \subset N_{a+b}$, and just \emph{homogeneous} if~$a = 0$.
Note that the grading of a homogeneous element is uniquely determined if the element is non-zero;
the analogue is true for the degree of a homogeneous homomorphism.

Direct sums and tensor products of graded modules inherit a grading in a straight forward way.
The \emph{dual} of a graded $R$\nbd module $M$, denoted by~$M^*$, is defined as $\Hom_R(M,R)$,
with grading given by setting $(M^*)_a$ to be the 
homogeneous homomorphisms $M\to R$ of degree~$a$.
The \emph{dual} of a homomorphism $f\colon M\to N$ is the homomorphism $f^*\colon N^*\to M^*$
given by precomposition with $f$. If $f$ is homogeneous of degree $a$, then $f^*$ is homogeneous of degree $-a$.

The \emph{degree shift by $a\in A$} of a graded $R$\nbd module $M$ is the module $N$ with
$N_{a+b} = M_b$. A graded module $R$\nbd module is called \emph{graded free} if it
has a basis consisting of homogeneous elements, or equivalently, if it is isomorphic
to a sum of degree shifts of $R$.

We will usually call $\Z$\nbd gradings by a variable name, e.g.~$q$\nbd grading.
We then write $\qdeg x \in \Z$ for the grading of a homogeneous non-zero element~$x$ and $\qdeg f$
for the degree of a homogeneous non-zero homomorphism~$f$.
We also write the degree shift of an $R$\nbd module by $n$ as~$q^nM$.
For a graded free module $M$ isomorphic to a finite direct sum~$\bigoplus_{j} q^{a_{j}}R$,
we define the \emph{graded rank} of $M$, denoted $\qrk M$, as the Laurent polynomial~$\sum_j q^{a_j} \in \Z_{\geq 0}[q^{\pm 1}]$.
Note that the evaluation at $q=1$ recovers the ungraded (i.e.~the usual) rank.
The graded rank is well-defined if the $\Z$\nbd graded ring $R$ has no non-trivial homogeneous elements of negative grading
(otherwise, it may e.g.~happen that $q^a R$ and $q^b R$ are isomorphic graded $R$\nbd modules for any~$a,b\in\Z$). Note that for graded free modules $M$ and $N$,
we have $\qrk(M \otimes_R N) = \qrk(M) \cdot \qrk(N)$, $\qrk(M \oplus N) = \qrk(M) + \qrk(N)$ and $\qrk(M^*)(t) = \qrk(M)(t^{-1})$.

In the case of a $\Z^2$\nbd grading, we will usually use two variable names, e.g.~$t$ and $q$.
Then, $\tqdeg \in \Z^2$ and $\tqrk \in \Z_{\geq 0}[q^{\pm 1},t^{\pm 1}]$ are defined similarly as above.

\begin{example}\label{ex:zg}
For our version of Khovanov homology,
we need the $\Z$\nbd graded ring needed~$\Z[G]$,
the integer polynomial ring in one formal variable~$G$,
with a grading called \emph{quantum grading} or $q$\nbd grading.
Note that setting $\qdeg G = -2$ determines $\qdeg$ completely.
The map $f\colon q^a\Z[G]\longrightarrow q^b\Z[G]$ given by multiplication with
$mG^c$ for $m,c\in\Z, m\neq 0, c\geq 0$ satisfies
\[
\qdeg f = -a + b - 2c,
\]
since it sends the homogeneous element $1\in q^a\Z[G]$ with $\qdeg 1 = a$
to the homogeneous element $mG^c\in q^b\Z[G]$ with~$\qdeg mG^c = b - 2c$.
This formula for $\qdeg f$ will often be implicitly used throughout the text.
\end{example}

For a $\Z$\nbd graded ring $R$, a \emph{grading} on a chain complex $(C,d)$ over $R$ is a grading of its chain modules~$C_i$,
such that the differential $d$ (which by our convention goes from $C_i$ to $C_{i+1}$) is a homogeneous map.
We will take the following alternative (but equivalent) perspective.
Denote the grading on $R$ by $q$, and endow $R$ with another $\Z$\nbd grading, denoted by $t$,
by putting all of $R$ in $t$\nbd degree $0$. In total, $R$ has become a $\Z^2$\nbd graded ring.
Now, a graded chain complex $C$ is just a graded $R$\nbd module
with a homogeneous endomorphism $d$ of $tq$\nbd degree $(1,0)$ and $d^2 = 0$.
If $C$ is graded free, then its \emph{total rank} is just defined as its graded rank.
A \emph{homogeneous chain map of degree $a$} $f\colon C\to C'$ between graded chain complexes $C, C'$
is a homogeneous homomorphism of $tq$\nbd degree $(0,a)$ such that $f\circ d_C = d_{C'}\circ f$.
Similarly, a \emph{homogeneous homotopy of degree $a$} $h\colon C\to C'$ 
is a homogeneous homomorphism of $tq$\nbd degree $(-1,a)$.
Two homogeneous chain maps $f,g\colon C\to C'$ of arbitrary degrees
are called \emph{homogeneously homotopic}, denoted by $f\simeq g$,
if there exists a homogeneous homotopy $h$ such that $f - g \simeq h\circ d_C + d_{C'}\circ h$.

The \emph{dual} of a graded chain complex $(C,d)$ is the graded chain complex $(C^*, d^*)$
with $C^*_i = (C_{-i})^*$ and $d^*\colon C^*_i \to C^*_{i+1}$ equal to the dual of $d\colon C_{-i-1}\to C_{-i}$.
Tensor products of graded chain complexes $C, C'$ over $R$ are defined as
$(C\otimes C')_i = \sum_{j+k=i} C_j\otimes C'_k$, with differential
(obeying the so-called Koszul sign convention)
\[
d_{C\otimes C'}(x\otimes y) = d_C(x)\otimes y + (-1)^{\tdeg x} x\otimes d_{C'}(y)
\]
for homogeneous elements $x\in C, y\in C'$.

We write $H(C)$ for the total homology of $C$, which, just as $C$, we see as a $\Z^2$\nbd graded $R$\nbd module.
If $M$ is $q$\nbd graded, then $H(C\otimes_R M)$ is $tq$\nbd graded;
if $M$ is not graded, then $H(C\otimes_R M)$ is just $t$\nbd graded.

\section{Khovanov homology as a black box}\label{sec:axioms}
We use a certain `universal' version of Khovanov homology,
which associates with every diagram $D$ of an oriented knot $K$ a graded chain complex $C(D)$ over the graded ring~$\Z[G]$ introduced in \cref{ex:zg}.
Let us now list some properties of this homology theory.
These properties do not determine the homology of any knot except the unknot.
Rather, they are `minimal requirements' 
for the construction of the Rasmussen invariants and our new invariant.
Here and throughout the text,
tensor products are taken over~$\Z[G]$ (unless indicated otherwise),
and $Z$ denotes the
ungraded $\Z[G]$\nbd module~$\Z[G]/(G-1)$.%
\medskip

\begin{enumerate}
\item \label{axiom:admissible} $C(D)$ consists of graded free modules with even degree shifts, and $C(D)$ is of finite total rank.
\item \label{axiom:knot} $H_i(C(D)\otimes Z)$ is isomorphic to $Z$ for $i=0$ and trivial for~$i\neq 0$.
\item \label{axiom:rm} If the diagrams $D_1$ and $D_2$ are related by a Reidemeister move, then
$C(D_1)$ and $C(D_2)$ are homogeneously homotopy equivalent.
\item \label{axiom:unknot} For $D$ the crossingless diagram,~$C(D) = t^0q^0\Z[G]$.
\item \label{axiom:sum} $C(D \# D')$ is isomorphic to~$C(D) \otimes C(D')$
\item \label{axiom:mirror} $C(-D)$ is isomorphic to the dual~$C(D)^*$ of $C(D)$.
\item \label{axiom:cobo} Let $K_1, K_2$ be knots with respective diagrams $D_1, D_2$.
If there exists a smooth connected cobordism of genus $g$ between $K_1$ and~$K_2$,
then there exists a homogeneous chain map $C(D_1)\to C(D_2)$
of quantum degree~$-2g$ that induces an isomorphism $H_0(C(D_1)\otimes Z) \to H_0(C(D_2)\otimes Z)$.
\end{enumerate}
\smallskip

\begin{definition}\label{def:knotlike}
We call a graded chain complex $C(D)$ over the graded ring $\Z[G]$
\emph{admissible} if it satisfies \eqref{axiom:admissible},
and \emph{knot-like} if it satisfies \eqref{axiom:admissible} and \eqref{axiom:knot}.
A homogeneous chain map inducing an isomorphism on $H_0(\,\cdot\,\otimes Z)$,
such as the one in \eqref{axiom:cobo}, will be called a \emph{$Z$\nbd isomorphism}.
\end{definition}

\section{A glimpse into the black box}\label{sec:glimpse}

The version of Khovanov homology that we are using here goes back to Naot~\cite{naotphd,zbMATH05118580}. It is also discussed at length in \cite[Section~3]{KWZ} and \cite[Section~2]{ilm21}.
Here, we will only briefly summarize how $C(D)$ may be constructed and why (1)--(7) hold.
To follow this subsection, some knowledge of the technicalities of Khovanov homology is required, in particular the papers~\cite{bncob,zbMATH05039154}. The non-expert reader is invited to skip to~\cref{sec:construction}.

There are two ways to construct the chain complex $C(D)$ for a given knot diagram $D$.
The first way is to cut $D$ open at a point away from the crossings, to obtain a 2-ended tangle diagram~$T$. Bar-Natan's construction \cite{bncob} associates with $T$ a chain complex over a certain category $\operatorname{Kob}$, whose objects are formal sums of grading shifted crossingless 2-ended tangles, and whose morphisms are matrices of $\Z$\nbd linear combinations of cobordisms up to isotopy, modulo certain local relations. As it turns out, the category $\operatorname{Kob}$ is equivalent to the category of graded free modules of finite rank over~$\Z[G]$, and homogeneous homomorphisms~\cite[Proposition~2.8]{ilm21}.
Under this equivalence the module $\Z[G]$ corresponds to a crossingless circleless 2-ended tangle (i.e.~simply an interval), multiplication by $G$ corresponds to the connected cobordism of genus one between two copies of this tangle, and Bar-Natan's complex associated with $T$ corresponds to $C(D)$.
The second way to obtain $C(D)$ is as the chain complex coming from the Frobenius algebra $\Z[G,X] / (X^2 + GX)$ over the ring~$\Z[G]$, reduced at the root~$X = 0$.

The properties (1)--(5) of $C(D)$ are known to hold in either of the two interpretations of $C(D)$~\cite{bncob,zbMATH05039154}. Property \eqref{axiom:knot} is e.g.~shown in~\cite{zbMATH07182539}:
$C(D) \otimes Z$ is isomorphic to the chain complex coming from the Frobenius algebra
$\Z[X] / (X^2 + X)$ over $\Z$, which is \emph{diagonalizable}: the basis $(a,b) = (X+1,-X)$
has the property that $a^2=a, b^2=b$ and $ab=0$. Indeed, a cycle whose class generates $H_0(C(D)\otimes Z)$
may be explicitly given. It is supported in a single vertex of the cube of resolutions,
namely the vertex corresponding to the oriented resolution.
Labeling all components of the oriented resolution by $a$ and $b$ such
that the unique interval component is colored by $a$, and any two adjacent components are colored differently yields the desired cycle.

Property (7) requires only a weak form of functoriality; in fact, the stronger statement is true that any smooth cobordism induces a chain map $C(D_1)\to C(D_2)$ whose homogeneous homotopy type up to sign depends only on the isotopy class of the cobordism~\cite{bncob}, and this construction is well-behaved with respect to composition. To see that $H_0(C(D_1)\otimes Z)\to H_0(C(D_2)\otimes Z)$ is an isomorphism, one proves that the cycle generating $H_0(C(D_1)\otimes Z)$ constructed above is sent to $\pm$ the corresponding cycle of $H_0(C(D_2)\otimes Z)$. This may be achieved by checking this statement for each Reidemeister move and each Morse move, similarly as in the proof of \cite[Proposition~4.1]{ras3}.

\section{Rasmussen invariants over fields}\label{sec:construction}
To warm up, let us see how to extract the well-known Rasmussen invariants over fields from our Khovanov chain complex $C(D)$, merely using the properties (1)--(7).
The following `structure theorem of Khovanov homology' goes back to~\cite{zbMATH05039154} and has since appeared in various versions (see e.g.~\cite{morrison},~\cite[Section 3.3.2]{KWZ}).%
\begin{proposition}\label{prop:structureF}
For all fields $\mathbb{F}$ and knot diagrams $D$, the following hold.
\begin{enumerate}[label=(\roman*)]
\item The chain complex $C(D) \otimes \mathbb{F}[G]$ decomposes as a sum
of a complex with trivial homology and
\[
t^0q^{s} \mathbb{F}[G] \quad \oplus\quad
\bigoplus_i t^{a_i} q^{b_i} \mathbb{F}[G] \xrightarrow{G^{c_i}} t^{a_i+1} q^{b_i+2c_i} \mathbb{F}[G],
\]
where $s\in 2\Z$, $a_i\in \Z$, $b_i \in 2\Z$, $c_i \in \Z_{>0}$.
\item The numbers $s, a_i, b_i, c_i$ are knot invariants.
\item For fixed $D$, the numbers $s, a_i, b_i, c_i$ are determined by $\Char\mathbb{F}$.
\end{enumerate}
\end{proposition}
For the proof of the proposition, we require the following lemma,
which may be paraphrased as: knot-like complexes are `knot-like for all coefficients'.%
\begin{lemma}\label{lemma:knotlikefield}
Let $D$ be a knot diagram and $M$ any abelian group.
Then $H_i(C(D)\otimes M[G]/(G-1))$ is isomorphic to $M[G]/(G-1)$ for $i = 0$
and trivial otherwise.
\end{lemma}
\begin{proof}
In this proof only, let us denote by $\mathcal{U}$ the forgetful functor from $\Z[G]$\nbd modules to abelian groups, which associates with a $\Z[G]$\nbd module its underlying $\Z$\nbd module.
We have $H_i(\mathcal{U}C(D)) \cong \mathcal{U}H_i(C(D))$. Hence, by \eqref{axiom:knot},
$H_i(\mathcal{U}C(D))$
is trivial for $i \neq 0$ and isomorphic to $\mathcal{U}Z \cong \Z$ for $i = 0$.
Since $H(\mathcal{U}C(D))$ is torsion free, the universal coefficient theorem implies
that $H_i(\mathcal{U}C(D)\otimes_{\Z} M) \cong H_i(\mathcal{U}C(D)) \otimes_{\Z} M$.
Now we conclude:
\begin{align*}
\mathcal{U} H_i(C(D)\otimes M[G]/(G-1)) & \cong 
H_i(\mathcal{U}C(D)\otimes_{\Z}\mathcal{U}M[G]/(G-1)) \\
& \cong H_i(\mathcal{U}C(D)\otimes_{\Z} M) \\
& \cong H_i(\mathcal{U}C(D)) \otimes_{\Z} M,
\end{align*}
which is isomorphic to $M$ for $i = 0$ and trivial otherwise.
Since $G$ operates as $1$ on $H_i(C(D)\otimes M[G]/(G-1))$, the statement follows.
\end{proof}

\begin{proof}[Proof of \cref{prop:structureF}]
\textbf{(i):} The existence of such a decomposition, but allowing arbitrarily many summands of the form $t^u q^s \mathbb{F}[G]$, follows using the Smith normal form, since,  by~(1),
$C(D) \otimes \mathbb{F}[G]$ is a chain complex of free modules over the PID~$\mathbb{F}[G]$ of finite total rank.
Note that each summand $t^u q^s \mathbb{F}[G]$ results in an $\mathbb{F}[G]/(G-1)$--summand in $H_u(C(D)\otimes \mathbb{F}[G]/(G-1))$. But by \cref{lemma:knotlikefield} with $M = \mathbb{F}$, $H_u(C(D)\otimes \mathbb{F}[G]/(G-1))$ is isomorphic to $\mathbb{F}[G]/(G-1)$ if $u = 0$ and trivial otherwise. This implies that there is a unique summand $t^u q^s \mathbb{F}[G]$, and it satisfies $u = 0$.

\textbf{(ii):} By \eqref{axiom:rm}, any two diagrams $D_1, D_2$ of a fixed knot have homogeneously homotopy equivalent complexes $C(D_1) \simeq C(D_2)$, and thus isomorphic graded homologies $H(C(D_1)\otimes \mathbb{F}[G]) \cong H(C(D_2)\otimes \mathbb{F}[G])$. Observe that the numbers $s, a_i, b_i, c_i$ are determined by
$H(C(D_j) \otimes \mathbb{F}[G])$.

\textbf{(iii):} Let $\mathbb{E}$ be the unique minimal subfield of~$\mathbb{F}$.
From $(C(D) \otimes \mathbb{E}[G]) \otimes \mathbb{F}[G] \cong C(D) \otimes \mathbb{F}[G]$
it follows that the numbers $s, a_i, b_i, c_i$ agree for $\mathbb{E}$ and~$\mathbb{F}$.
Now the statement follows since any two fields of the same characteristic
have isomorphic minimal subfields.%
\end{proof}
\begin{definition}
For a knot $K$ with diagram~$D$, the number $s$ of~\cref{prop:structureF}
is called the \emph{Rasmussen invariant over $\mathbb{F}$ of $K$},
denoted by $s_c(K)$ for $c = \Char\mathbb{F}$.
\end{definition}
\begin{proposition}\label{prop:rasmussenF}
The Rasmussen invariant $s_c$ has the following properties:
\begin{enumerate}[label=(\roman*)]
\item $s_c(K_1 \# K_2) = s_c(K_1) + s_c(K_2)$.
\item $s_c(-K) = -s_c(K)$.
\item If there exists a smooth connected cobordism of genus $g$ between the knots $K_1$ and $K_2$,
then $|s_c(K_1) - s_c(K_2)| \leq 2g$.
\end{enumerate}
In particular, $s_c$ induces a homomorphism $\mathcal{C}\to 2\Z$.
\end{proposition}
\begin{proof}
\textbf{(i):} Let $D_i$ be a diagram for $K_i$. By \cref{prop:structureF}, $C(D_i)$ admits a summand with graded rank $t^0q^{s(K_i)}$.
It follows that $C(D_1) \otimes C(D_2)$ admits a summand with graded rank $t^0q^{s(K_1)} \cdot t^0q^{s(K_2)} = t^0 q^{s(K_1) + s(K_2)}$. But by \eqref{axiom:sum}, $C(D_1) \otimes C(D_2) \cong C(D_1 \# D_2)$,
and so this chain complex admits a summand with graded rank $t^0 q^{s(K_1 \# K_2)}$.
Since \cref{prop:structureF} says that there is a unique summand of ungraded rank~$1$,
it follows that $t^0 q^{s(K_1) + s(K_2)} = t^0 q^{s(K_1 \# K_2)}$ and thus $s(K_1) + s(K_2) = s(K_1 \# K_2)$.

\textbf{(ii):} One reasons similarly as in (i). Let $D$ be a diagram of $K$. By \eqref{axiom:mirror}, $C(-D) \cong C(D)^*$. Since the dual of a sum is the sum of the duals, $C(-D)$ admits a summand with graded rank $t^0 q^{-s(K)}$, which must equal $t^0 q^{s(-K)}$ because of the uniqueness of the summand with ungraded rank~$1$.

\textbf{(iii):} If two knots $K_1$ and~$K_2$ with respective diagrams $D_1$ and~$D_2$
are related by a smooth connected cobordism of genus~$g$, then by~\eqref{axiom:cobo},
there exists a $Z$\nbd isomorphism $f\colon C(D_1)\to C(D_2)$
with $\qdeg f = -2g$.
Pick a decomposition as in \cref{prop:structureF} for $C(D_1)\otimes \mathbb{F}[G]$ and $C(D_2)\otimes \mathbb{F}[G]$.
With respect to these decompositions, we may write $f\otimes \mathbb{F}[G]$ as a matrix.
Consider the matrix entry $\alpha$ that is a homomorphism
$t^0q^{s_c(K_1)}\mathbb{F}[G] \to t^0 q^{s_c(K_2)}\mathbb{F}[G]$.
Since $\alpha$ is homogeneous, we have $\alpha = m\cdot G^n$ for some $n\geq 0$ and $m\in\Z$.
By \cref{lemma:knotlikefield}, $f$ also induces an isomorphism 
\[
H_0(C(D_1) \otimes \mathbb{F}[G]/(G-1))\to H_0(C(D_2) \otimes \mathbb{F}[G]/(G-1)),
\]
or equivalently $t^0 s^{s(D_1)} \mathbb{F}[G]/(G-1) \to t^0 s^{s(D_2)} \mathbb{F}[G]/(G-1)$.
Since this isomorphism is induced by $\alpha$, it follows that $\alpha \neq 0$.
Because $\qdeg \alpha = -2g$, we find $-2g = s(D_2) - s(D_1) - 2n$.
And since $n \geq 0$, we have $s(D_1) - s(D_2) \leq 2g$.
By turning the cobordism upside down and switching the roles of $D_1$ and $D_2$,
we also get $s(D_2) - s(D_1) \leq 2g$, and so $|s(K_1) - s(K_2)| \leq 2g$ follows.
\end{proof}
To summarize, the method we used here to extract Rasmussen's invariant over a field was to identify a `special summand' in the direct sum decomposition of the Khovanov chain complex. The Rasmussen invariant is precisely the information encoded by the graded isomorphism type of this summand, which has graded rank $t^0q^{2s_c(K)}$.
In contrast, Rasmussen's original method was to extract the invariant from the limit of the spectral sequence induced by Lee's filtered version of the Khovanov chain complex.
That these two approaches yield the same result can easily be seen from \cref{prop:structureF}.
\section{The new homomorphism \texorpdfstring{$S$}{S}}\label{sec:rasint}%
How best to extract a Rasmussen invariant directly from the chain complex $C(D)$ over $\Z[G]$ without tensoring with a field? Sch\"utz's invariant $s^{\Z}$ may be extracted from the limit of the spectral sequence associated with the filtration on $C(D) \otimes Z$, in analogy to Rasmussen's original method.
Another tempting approach is the following:
decompose $C(D)$ as a sum of indecomposable summands;
it will follow from \cref{lem:summand} below that in any such decomposition there is a unique summand with odd Euler characteristic. Then one might hope that the isomorphism type of this `special summand' is a concordance invariant. However, this seems doubtful since
the decomposition of $C(D)$ as a sum of indecomposables is not essentially unique (see \cite[Remark~3.19]{ilm21}).

Here, we implement a third approach, namely to consider an equivalence relation on chain complexes $C(D)$ that is just coarse enough to guarantee that concordant knots have equivalent chain complexes.
Property \eqref{axiom:cobo} of the complexes $C(D)$ makes it almost obvious how that equivalence relation needs to be defined.
\begin{definition}\label{def:Zeq}
Let us call two admissible chain complexes $C$ and $C'$ \emph{$Z$\nbd equivalent}, denoted by $C\sim C'$, if there exist $Z$\nbd isomorphisms $f\colon C\to C'$ and $g\colon C'\to C$
of quantum degree $0$.
\end{definition}
See \cref{def:knotlike} for the meaning of $Z$\nbd isomorphism.
Let us first show that for all knots $K$, $J$ and all $c$ equal to zero or a prime,
$C(K) \sim C(J)$ implies that $s_c(K) = s_c(J)$. That constitutes the\ldots
\begin{proof}[\ldots first half of the proof of \cref{thm:stronger}]\hspace{-8pt},
which bears some similarity with the proof of \cref{prop:rasmussenF}(iii).
Let $f\colon C(K)\to C(J)$ be a $Z$\nbd isomorphism of quantum degree $0$
and let $\mathbb{F}$ be a field of characteristic $c$.
Tensoring results in a chain map $f\otimes \mathbb{F}\colon C(K)\otimes\mathbb{F}\to C(J)\otimes\mathbb{F}$. 
By \cref{prop:structureF},
the complexes $\mathbb{F}\otimes C(K)$ and $\mathbb{F}\otimes C(J)$
admit summands $t^0 q^{s_c(K)} \mathbb{F}[G]$
and $t^0 q^{s_c(J)} \mathbb{F}[G]$, respectively.
The chain map $f\otimes \mathbb{F}$ induces
a map $\alpha\colon t^0 q^{s_c(K)} \mathbb{F}[G] \to t^0 q^{s_c(J)} \mathbb{F}[G]$,
which is non-zero because,
using \cref{lemma:knotlikefield}, $\alpha$ induces isomorphisms $H(C \otimes \mathbb{F}[G]/(G-1)) \to H(C' \otimes \mathbb{F}[G]/(G-1))$.
From $\alpha\neq 0$ it follows that $s_c(J) \geq s_c(K)$.
Reversing the roles of $K$ and $J$ yields $s_c(K) \geq s_c(J)$.
Overall we find $s_c(J) = s_c(K)$.
\end{proof}
The following lemma shows that in spite of their non-uniqueness,
direct sum decompositions of complexes are still useful in the context of $Z$\nbd equivalence.
Indeed, to specify a $Z$\nbd equivalence class, we will usually give an indecomposable
representative (see e.g.~\cref{def:staircase}, \cref{ex:18}, \cref{ex:staircasetimesdual}).%
\begin{lemma}\label{lem:summand}
Let $C$ be a knot-like complex.
\begin{enumerate}[label=(\roman*)]
\item If $C\sim C' \oplus C''$, then exactly one of $C'$ and $C''$ is knot-like.
\item If $C\sim C' \oplus C''$ and $C'$ is knot-like, then $C\sim C'$.
\item There exists a knot-like complex $C'$ such that $C\sim C'$
and $C'$ is indecomposable, i.e.~not isomorphic to the sum of two non-trivial admissible complexes.
\end{enumerate}
\end{lemma}
\begin{proof}
\textbf{(i):}
Since $H(C\otimes Z) \cong H(C'\otimes Z) \oplus H(C''\otimes Z)$ and $Z$ cannot be written as sum of two non-trivial $\Z[G]$\nbd modules, it follows that $H(C'\otimes Z) \cong t^0 Z$ and $H(C''\otimes Z) \cong 0$, or vice versa.

\textbf{(ii):}
Similarly as in (i), it follows that $C''\otimes Z$ has trivial homology.
Let $f\colon C'\oplus C''\to C'$ and $g\colon C'\to C'\oplus C''$ be the projection and inclusion map of the summand, respectively. Then 
$f_*\colon H((C'\oplus C'')\otimes Z) \to H(C'\otimes Z)$
and 
$g_*\colon H(C'\otimes Z) \to H((C'\oplus C'')\otimes Z)$ are isomorphisms,
and so $C'\oplus C'' \sim C'$. Transitivity now implies $C\sim C'$.

\textbf{(iii):} Using (i) and (ii), this follows by induction over the total rank of~$C$.
\end{proof}
The proof that dual complexes constitute multiplicative inverses will rely on the following lemma,
which we formulate in a more general setting.
\begin{lemma}\label{lem:inverse}
Let $R$ be a $\Z$\nbd graded ring, and $C$ a chain complex over $R$ of graded free $R$\nbd modules
of finite total rank. If the ungraded Euler characteristic $\chi(C)$ of $C$ is a unit in $R$,
then $C\otimes_R C^*$ admits a summand $t^0q^0 R$, where as usual $t$ denotes homological grading,
and $q$ denotes the grading on $R$.
\end{lemma}
\begin{proof}
The existence of such a summand is by the splitting lemma equivalent
to the existence of homogeneous chain maps $f\colon t^0q^0R \to C\otimes_R C^*$ and $g\colon C\otimes_R C^* \to t^0q^0R$ such that $g\circ f = \id_{t^0q^0R}$.
Let us first pick bases for the involved chain complexes, then we will construct such maps $f$ and $g$.

Pick a homogeneous basis $(x_i)_i$ for $C$, and denote by $(x_i^*)_i$ the dual basis for $C^*$,
which satisfies $x_i^*(x_j) = \delta_{ij}$ (where $\delta_{ij}$ is the Kronecker delta).
Note that $\tqdeg x_i^* = -\tqdeg x_i$.
Moreover, the differentials of $C$ and $C^*$ are related by
\[
d_{C^*}(x_j^*)(x_i) = x_j^*(d_C(x_i)),
\]
which is just another way of saying that the matrix of $d_{C^*}$ with respect to the basis $(x_i^*)_i$ equals the transpose of the matrix of $d_C$ with respect to the basis $(x_i)_i$.
Now, the elements $(y_{ij})_{i,j}$ given as $y_{ij}\coloneqq x_i\otimes x_j^*$ form a homogeneous basis of $C\otimes_R C^*$. Denote by $(y_{ij}^*)_{i,j}$ its dual basis for $(C\otimes_R C^*)^*$;
so $y_{ij}^*$ is a homogeneous homomorphism (but not in general a chain map)
$C\otimes_R C^* \to t^0q^0R$. We have
$\tqdeg y_{ij}^* = -\tqdeg y_{ij} = -\tqdeg x_i + \tqdeg x_j$.
Also, note that for $v\in C, w\in C^*$, we have
have $y^*_{ij}(v\otimes w) = x_i^*(v) \cdot w(x_j)$.

Let us define $f\colon t^0q^0R \to C\otimes_R C^*$ and $g\colon C\otimes_R C^* \to t^0q^0R$ by%
\[
f(1) = \sum_i (-1)^{\zeta\tdeg x_i} y_{ii}\quad\text{and}\quad
g = \frac{1}{\chi(R)}\sum_i (-1)^{\tdeg x_i + \zeta\tdeg x_i}y^*_{ii}.
\]
Here, $\zeta\colon\Z\to\{0,1\}$ is defined as $\zeta m = 0$ if $m \equiv 0\pmod{4}$ or $m \equiv 1\pmod{4}$, and $\zeta m = 1$ otherwise.\footnote{If we had used the alternative sign convention $d(x\otimes y) = d(x)\otimes y + (-1)^{\tdeg x + \tdeg y} x\otimes d(y)$ for the tensor product of chain complexes, the signs in the definition of $f$ and $g$ would have been simpler and $\zeta$ not be required.}
We will now check that, as desired, $f$ and $g$ are homogeneous chain maps such that $g\circ f = \id_{t^0q^0R}$.

First, $\tqdeg y_{ii} = \tqdeg y^*_{ii} = (0,0)$, so $f$ and $g$ are homogeneous.
Next,
\begin{align*}
g(f(1)) & = \frac{1}{\chi(R)}\sum_i (-1)^{\tdeg x_i + \zeta\tdeg x_i}y^*_{ii}\Bigl(\sum_j (-1)^{\zeta\tdeg x_j} y_{jj}\Bigr) \\
        & = \frac{1}{\chi(R)}\sum_{i,j} (-1)^{\tdeg x_i+\zeta\tdeg x_i + \zeta\tdeg x_j} \delta_{ij} \\
        & = \frac{1}{\chi(R)}\sum_i (-1)^{\tdeg x_i} = \frac{1}{\chi(R)}\cdot\chi(R) = 1
\end{align*}
as desired.
For $f$ to be a chain map is equivalent to $d_{C\otimes_R C^*}(f(1)) = 0$,
which is equivalent to $y_{jk}^*(d_{C\otimes_R C^*} (f(1))) = 0$ for all $j,k$.
One computes
\begin{align*}
& y_{jk}^*(d_{C\otimes_R C^*} (f(1)))  = \sum_i (-1)^{\zeta\tdeg x_i}\cdot y_{jk}^*(d_{C\otimes_R C^*}(x_i \otimes x_i^*)) \\
 & = \sum_i (-1)^{\zeta\tdeg x_i} \cdot \bigl(y_{jk}^*(d_C(x_i) \otimes x_i^*) + (-1)^{\tdeg x_i}\cdot y_{jk}^*( x_i \otimes d_{C^*}(x_i^*))\bigr) \\
 & = \sum_i (-1)^{\zeta\tdeg x_i} \cdot x_j^*(d_C(x_i)) \cdot \delta_{ki} + (-1)^{\tdeg x_i+\zeta\tdeg x_i}\cdot \delta_{ji} \cdot d_{C^*}(x_i^*)(x_k) \\
 & = (-1)^{\zeta\tdeg x_k} x_j^*(d_C(x_k)) + (-1)^{\tdeg x_j+\zeta\tdeg x_j}d_{C^*}(x_j^*)(x_k) \\
 & = ((-1)^{\zeta\tdeg x_k} + (-1)^{\tdeg x_j+\zeta\tdeg x_j}) \cdot x_j^*(d_C(x_k)).
\end{align*}
Since $\tdeg d_C(x_k) = \tdeg(x_k) + 1$, the factor $x_j^*(d_C(x_k))$ vanishes if $\tdeg(x_j) \neq \tdeg(x_k) + 1$. If $\tdeg(x_j) = \tdeg(x_k) + 1$, we claim that the first factor vanishes.
Let $\tdeg(x_j) = m$. It suffices to check for all integers $m$
that $\zeta(m - 1) \not\equiv m+\zeta(m)\pmod{2}$.
This may be done by going through the possible remainders of $m$ mod 4.
This concludes the proof that $f$ is a chain map.
To show that $g$ is, too, we need to check that $g(d(y_{jk})) = 0$ for all~$j,k$,
or equivalently $\chi(R)\cdot g(d(y_{jk})) = 0$.
Similarly as before, one computes
\begin{align*}
& \chi(R)\cdot g(d(y_{jk})) = \sum_i (-1)^{\tdeg x_i + \zeta \tdeg x_i} y^*_{ii}(d(x_j\otimes x_k^*)) \\
& = \sum_i (-1)^{\tdeg x_i + \zeta \tdeg x_i} y^*_{ii}(d(x_j)\otimes x_k^* + (-1)^{\tdeg x_j} x_j \otimes d_{C^*}(x_k^*)) \\
& = (-1)^{\tdeg x_k + \zeta \tdeg x_k} x_k^*(d(x_j)) + (-1)^{\tdeg x_j} d_{C^*}(x_k^*)(x_j) \\
& = ((-1)^{\tdeg x_k + \zeta \tdeg x_k} + (-1)^{\tdeg x_j})\cdot x_k^*(d(x_j)).
\end{align*}
But this is just the term (with $j$ and $k$ switched)
we already encountered and showed to be 0 in the previous calculation.
\end{proof}
\cref{thm:construction} now follows from the following.
\begin{proposition}
\begin{enumerate}[label=(\roman*)]
\item $Z$\nbd equivalence is an equivalence relation on the set of admissible complexes.
\item A semiring\footnote{Like a ring, but without required additive inverses.}
 $\mathcal{R}$  is given
by the set of $Z$\nbd equivalence classes of admissible chain complexes,
equipped with $\oplus$ as addition with the trivial complex as neutral element,
and $\otimes$ as multiplication with $t^0q^0\Z[G]$ as neutral element.
\item For the class of a knot-like complex $C$, a multiplicative inverse in $\mathcal{R}$ is given by the class of~$C^*$.
\item Let $\mathcal{G}\subset\mathcal{R}$ consist of the classes of knot-like complexes $C$. Then an abelian group is given by $\mathcal{G}$ with group operation $\otimes$ and neutral element $t^0q^0\Z[G]$.
\end{enumerate}
\end{proposition}
\begin{proof}
Claims (i) and (ii) are straight forward.
To show claim (iii), note that a knot-like complex $C\in\mathcal{R}$
has $\chi(C) = \chi(C\otimes Z) = \chi(H(C\otimes Z)) = 1$.
So \cref{lem:inverse} implies that $C\otimes C^*$ admits a summand $t^0q^0\Z[G]$.
Since this summand is knot-like, \cref{lem:summand}(ii) implies $C\otimes C^* \sim t^0q^0\Z[G]$.
Now claim (iv) follows quickly.
\end{proof}

Let us point out that the proof of \cref{thm:construction} in this section has been almost automatic,
with the exception of showing that dual complexes are multiplicative inverses in \cref{lem:inverse}.

\section{Calculations of \texorpdfstring{$S$}{S}}\label{sec:calc}
Let us introduce and examine staircases: our principal examples of chain complexes with interesting $S$--invariant.
\begin{definition}\label{def:staircase}
Let $n \geq 0$ and $A = (a_1, \ldots, a_n)$ such that
$a_i \geq 0$ and $a_i$~divides\footnote{%
We follow the standard conventions:
on the one hand, every number divides $0$; on the other hand, $0$ divides no other number except $0$;
and finally, the only numbers coprime to $0$ are $\pm 1$.}
~$a_{i+1}$ for all~$i$.
Then the \emph{staircase}\footnote{Not to be confused with the staircase complexes in Floer homologies.} $\stair_A$ is the chain complex
\[
\Z[G]^{n+1} \xrightarrow{d_A} \Z[G]^n,
\]
with differential $d_A$ represented by the matrix
\[
\begin{pNiceMatrix}
     a_1 & G \\
     &\Ddots&\Ddots&  \\
     && a_n & G
\end{pNiceMatrix}
\]
with respect to a homogeneous basis $x_1, \ldots, x_{2n+1}$
with graded rank $t^0(q^{2n} + q^{2n-2} + \dots + q^0) + t^1(q^{2n} + q^{2n-2} + \dots + q^2)$.
\end{definition}
Note that for all $A$ as above, the complex $\stair_A$ is knot-like.
The name \emph{staircase} comes from the following depiction of the complex,
in which the zeroth and first chain module consist of the direct sum along
the first and second column, respectively:
\[
\newcommand{\myc}[1]{\makebox[3em]{\raisebox{0pt}[1.7ex][0.3ex]{\ensuremath{#1}}}}
\stair_A\ =
\begin{tikzcd}
\myc{q^{2n}\Z[G]} \ar[r,"a_{1}"]               & \myc{q^{2n}\Z[G]} \\
\myc{\vdots}      \ar[ru,"G"]                  & \myc{\vdots}   \\
\myc{q^4\Z[G]}    \ar[ru,"G"] \ar[r,"a_{n-1}"] & \myc{q^4\Z[G]} \\
\myc{q^2\Z[G]}    \ar[ru,"G"] \ar[r,"a_n"]     & \myc{q^2\Z[G]} \\
\myc{q^0\Z[G]}    \ar[ru,"G"] 
\end{tikzcd}
\]
\smallskip

The alert reader realizes they have already encountered staircases:
namely in \cref{conj:whitehead}, which says that $S(W^+_{3n}(T_{2,2n+1})) = [\Sigma_{(2^n)}]$.
\begin{remark}
Note that $\stair_{(1)} \sim \stair_{()}$ and $\stair_{(0)} \sim q^2 \stair_{()}$.
More generally, $\stair_{(1,\ldots,1,a_1, \ldots, a_n, 0,\ldots, 0)} \sim q^{2m} \stair_{a_1, \ldots, a_n}$,
where $m$ is the number of $a_i$ equal to zero.
Our main interest is in staircases with $a_i \geq 2$ for all $i$,
and we only allow $a_i\in\{0,1\}$ for the sake of completeness, and because it allows a simplification in the proof of \cref{lem:staircase3} below.
\end{remark}
In \cite[Section 4.1]{ilm21}, the special case of staircases $\stair_A$ with $a_1 = \dots = a_n = 2$ was considered (which was in that paper denoted by $S_n$), and it was shown that $\stair_A \sim C(K^{\# n})$ for $K = 14n19265$. Here, we generalize those results.
The following proposition completely describes the calculus of
the subgroup $\calstair\subset \mathcal{G}$ generated by staircases.
It immediately implies parts (i) and (ii) of \cref{thm:kernel}.
\begin{proposition}\label{prop:staircases}
\begin{enumerate}[label=(\roman*)]
\item For $A$ as in \cref{def:staircase}, we have \[\stair_A \sim \stair_{(a_1)} \otimes \dots \otimes \stair_{(a_n)}.\]
\item For pairwise coprime numbers $r_1, \ldots, r_m\geq 2$,
we have \[\stair_{(\prod_i r_i)} \sim \stair_{(r_1)} \otimes \dots \otimes \stair_{(r_m)}.\]
\item The abelian group $\calstair$ is isomorphic to $\Z^{\oplus\infty}$; indeed, a basis of $\calstair$ is given by $\stair_{(r)}$ where $r$ ranges over
the set $P = \{0,2,3,4,5,7,\ldots\}$
consisting of zero and all non-trivial prime powers.
\end{enumerate}
\end{proposition}
\begin{remark}
In other words, $\calstair$ is isomorphic to the Grothendieck group of the monoid of finitely generated abelian groups. The isomorphism is given as the linear extension of sending $\stair_A$ to 
$\Z/(a_1) \oplus \dots \oplus \Z/(a_n)$.
\end{remark}
\begin{proof}[Proof of \cref{prop:staircases}]
The meat of the proof is outsourced to three lemmas.
Claims (i) and (ii) follow by induction from \cref{lem:staircase1,lem:staircase2} below, respectively.
Those two lemmas are proven by explicit basis changes in the tensor product.
Let us now prove (iii). It follows from (i) and (ii) that $(\stair_{(r)})_{r\in P}$ generates $\calstair$.
To show linear independence, we have to show that if $\lambda_{r}\in\Z$ for all $r\in P$,
with only finitely many $\lambda_r$ non-zero, then
\[
\bigotimes_{r\in P} (\stair_{(r)})^{\otimes \lambda_{r}} \sim \stair_{()} \quad\Rightarrow\quad
\forall r: \lambda_{r} = 0.
\]
Here $(\stair_{(r)})^{\otimes \lambda_{r}}$ is understood as $(\stair^*_{(r)})^{\otimes -\lambda_{r}}$
for negative~$\lambda_r$.
To prove this, first bring all terms with negative $\lambda_{r}$ to the right-hand side of the above $\sim$ equation,
and then use (i) and (ii) to replace the left and right hand side of the equation by a $Z$\nbd equivalent term $\stair_{(A)}$ and $\stair_{(B)}$, respectively, such that 
$A = (a_1, \ldots, a_n), B = (b_1, \ldots, b_m)$ with $a_i\geq 0, a_i\neq 1, a_i|a_{i+1}$ and $b_i \geq 0, b_i\neq 1, b_i|b_{i+1}$.
This may be achieved similarly as one finds an isomorphism from a given finite sum of cyclic abelian groups to an abelian group $\Z/a_1 \oplus \dots \oplus \Z/a_n$ with $a_i$ as above.
It will be proven in \cref{lem:staircase3} that for such $A$ and $B$, $A\neq B$
implies $\stair_A \not\sim \stair_B$. That concludes the proof.
\end{proof}
\begin{lemma}\label{lem:staircase1}
For $A$ as in \cref{def:staircase}, we have
\[
\stair_{(a_1)} \otimes \stair_{(a_2, \dots, a_n)} \sim \stair_{(a_1, \dots, a_n)}.
\]
\end{lemma}
\begin{proof}
If $a_1$ is zero, then all $a_i$ are, since they are all divisible by $a_1$.
In that case, $\stair_{(a_1)} \sim q^2 \stair_{()}$, 
$\stair_{(a_2, \ldots, a_n)} \sim q^{2n-2} \stair_{()}$ and $\stair_A \sim q^{2n} \stair_{()}$,
so the desired statement follows easily. Hence we now assume $a_1 \neq 0$.
Let $x_1, x_2, x_3$ and $y_1, \ldots, y_{2n-1}$ be bases of $\stair_{a_1}$
and $\stair_{(a_2, \dots, a_n)}$, respectively, as in \cref{def:staircase}.
Then we use the following bases for $\stair_{(a_1)} \otimes \stair_{(a_2, \dots, a_n)}$:
\[
\begin{pNiceArray}{c}[margin]
\arrayrulecolor{darkred}
x_1\otimes y_1, \\
\vdots \\
x_1 \otimes y_{n}, \\\hline
x_2\otimes y_1,\\
\vdots\\
x_2 \otimes y_{n}
\end{pNiceArray}
\xrightarrow{d_0}
\begin{pNiceArray}{c}[margin]
\arrayrulecolor{darkred}
x_1\otimes y_{n+1}, \\
\vdots \\
x_1\otimes y_{2n-1}, \\\hline
x_2\otimes y_{n+1},\\
\vdots \\
x_2\otimes y_{2n-1}, \\\hline
x_3\otimes y_1,\\
\vdots\\
x_3 \otimes y_{n}
\end{pNiceArray}
\xrightarrow{d_1}
\begin{pNiceArray}{c}[margin]
x_3\otimes y_{n+1},\\
\vdots \\
x_3\otimes y_{2n-1}
\end{pNiceArray}
\]
Below, we write the differentials as matrices with respect to those bases.
Note that the bases in $t$--degree $0$, $1$, and $2$ are formed of two, three, and one `cluster', each.
For convenience, these clusters are hinted at by horizontal and vertical red lines.
The proof strategy is simply to apply basis changes to transform $\stair_{(a_1)} \otimes \stair_{(a_2, \dots, a_n)}$ into a complex that splits off $\stair_{(a_1, \dots, a_n)}$ as summand. Instead of writing down the basis change matrices, we indicate basis changes by blue arrows. Let us apply the following basis change to the original differentials of $\stair_{(a_1)} \otimes \stair_{(a_2, \dots, a_n)}$:
\[
\begin{pNiceArray}{*4{@{}c}|c*3{@{}c}}[margin]
\arrayrulecolor{darkred}
 a_2 & G &&& \\[-2mm]
  & \Ddots & \Ddots &&& \\[-2mm]
  && a_n & G \\\hline
 &&&& a_2 & G \\[-2mm]
 &&&&  & \Ddots & \Ddots \\[-2mm]
 &&&&  && a_n & G \\\hline
 a_1     &&&& G \\[-2mm]
  & \Ddots    &&& & \Ddots \\[-2mm]
  && \Ddots   &&& & \Ddots \\[-2mm]
  &&& a_1 && && G \\
\CodeAfter
        \tikz \draw[color=blue,thick,decorate,decoration={brace,mirror}] ([shift={(-1em,-0.1em)}]7-|1) -- ([shift={(-1em,0em)}]9.5-|1);
        \tikz \draw[color=blue,thick,decorate,decoration={brace,mirror}] ([shift={(-1em,-0.2em)}]1-|1) -- ([shift={(-1em,0em)}]3.5-|1);
        \tikz \draw[color=blue,thick,decorate,decoration={brace}] ([shift={(1em,-0.1em)}]8-|11) -- ([shift={(1em,0em)}]10.5-|11);
        \tikz \draw[color=blue,thick,decorate,decoration={brace}] ([shift={(1em,-0.2em)}]4-|11) -- ([shift={(1em,0em)}]6.5-|11);
        \tikz \draw[color=blue,thick,->,out=130,in=-130] ([shift={(-1.5em,-0.1em)}]8-|1) to node[right]{$\tfrac{-a_i}{a_1}$} ([shift={(-1.5em,-0.3em)}]2-|1);
        \tikz \draw[color=blue,thick,->,out=50,in=-50] ([shift={(1.5em,-0.1em)}]9-|11) to node[left]{$-1$} ([shift={(1.5em,-0.3em)}]5-|11);
\end{pNiceArray}
\qquad
\begin{pNiceArray}{*3{@{}c}|c*2{@{}c}|c*3{@{}c}}[margin]
\arrayrulecolor{darkred}
 -a_1 &&& -G & && a_2 & G \\[-2mm]
 & \Ddots &&& \Ddots &&& \Ddots & \Ddots \\[-2mm]
 && -a_1 &&& -G &&& a_n & G \\
\CodeAfter
        \tikz \draw[color=blue,thick,decorate,decoration={brace,mirror}] ([shift={(0em,-0.8em)}]3-|1) -- ([shift={(0em,-0.8em)}]3-|3.5);
        \tikz \draw[color=blue,thick,decorate,decoration={brace,mirror}] ([shift={(0em,-0.8em)}]3-|7.5) -- ([shift={(0em,-0.8em)}]3-|9.5);
        \tikz \draw[color=blue,thick,decorate,decoration={brace}] (1-|4.5) -- (1-|6.5);
        \tikz \draw[color=blue,thick,decorate,decoration={brace}] (1-|8.5) -- (1-|10.5);
        \tikz \draw[color=blue,thick,->,out=30,in=150] ([shift={(1em,1em)}]1-4.center) to node[above]{$1$} ([shift={(1em,1em)}]1-8.center);
        \tikz \draw[color=blue,thick,->,out=-30,in=-150] ([shift={(-1.2em,-1.2em)}]3-3.center) to node[below]{$\tfrac{a_i}{a_1}$} ([shift={(-1.2em,-1.2em)}]3-9.center);
\end{pNiceArray}.
\]
Here, we use that $a_1 \neq 0$ divides $a_i$. The basis change results in:
\[
\begin{pNiceArray}{*4{@{}c}|c*3{@{}c}}[margin]
\arrayrulecolor{darkred}
0 & G &&& -\tfrac{a_2}{a_1}G &  \\[-2mm]
  & & \Ddots &&& \Ddots &  \\[-2mm]
  &&  & G &&& -\tfrac{a_n}{a_1}G & 0 \\\hline
0 & -a_1 &&& a_2 &  \\[-2mm] &  & \Ddots &&  &   \Ddots \\[-2mm]
 &&  & -a_1 &  && a_n & 0 \\\hline
 a_1     &&&& G \\[-2mm]
  & \Ddots    &&& & \Ddots \\[-2mm]
  && \Ddots   &&& & \Ddots \\[-2mm]
 &&& a_1 & & && G \\
\CodeAfter
        \tikz \draw[color=blue,thick,->] plot [smooth, tension=1.7] coordinates {([shift={(.5em,0em)}]11-|1) ([shift={(.5em,-1.0em)}]11-|3) ([shift={(.5em,0em)}]11-|5)};
\end{pNiceArray}
\quad
\begin{pNiceArray}{*3{@{}c}|c*2{@{}c}|c}[margin]
\arrayrulecolor{darkred}
 -a_1 &&& -G  \\[-2mm]
 & \Ddots &&& \Ddots  \\[-2mm]
 && -a_1 &&& -G & \phantom{G} \\
\end{pNiceArray}.
\]
\vspace{.5em}

Now we apply two basis changes to the columns of the matrix; since they do not alter the second differential, we omit it. Firstly, permute the columns, moving the first column to the beginning of the second cluster, secondly add multiples of columns from the first cluster to the second:
\[
\begin{pNiceArray}{*3{@{}c}|c*4{@{}c}}[margin,name=m]
\arrayrulecolor{darkred}
 G &&& 0 & -\tfrac{a_2}{a_1}G &  \\[-2mm]
   & \Ddots &&&& \Ddots &  \\[-2mm]
  &  & G &&&& -\tfrac{a_n}{a_1}G & 0 \\\hline
 -a_1 &&& 0 & a_2 &  \\[-2mm]
   & \Ddots &&  &&   \Ddots \\[-2mm]
 &  & -a_1 &  &&& a_n & 0 \\\hline
 0     &&& a_1 & G \\[-2mm]
a_1  &     &&& & \\[-2mm]
 & \Ddots   &&& & \\[-2mm]
  && a_1 && &&& G \\
\CodeAfter
\tikz\draw[line width=0.3mm, shorten >=2.5mm, shorten <=2.5mm, loosely dotted] (m-7-5)--(m-10-8);
        \tikz \draw[color=blue,thick,decorate,decoration={brace,mirror}] ([shift={(0em,-0.8em)}]10-|1) -- ([shift={(0em,-0.8em)}]10-|4);
        \tikz \draw[color=blue,thick,decorate,decoration={brace,mirror}] ([shift={(0.7em,-0.8em)}]10-|5) -- ([shift={(0em,-0.8em)}]10-|9);
\tikz \draw[color=blue,thick,->,out=-30,in=-150] ([shift={(-0.2em,-1.2em)}]10-|3) to node[above]{$\tfrac{a_i}{a_1}$} ([shift={(-1.2em,-1.2em)}]10-|7.5);
\end{pNiceArray},
\]
\vspace{.5em}

\[
\begin{pNiceArray}{*3{@{}c}|c*4{@{}c}}[margin,name=m2]
\arrayrulecolor{darkred}
 G &&& \phantom{0} & \phantom{-\tfrac{a_2}{a_1}G} &  \\[-2mm]
   & \Ddots &&&&  &  \\[-2mm]
  &  & G &&&& \phantom{-\tfrac{a_n}{a_1}G} & \phantom{0} \\\hline
 -a_1 &&& \\[-2mm]
   & \Ddots && \\[-2mm]
 &  & -a_1 &  &&& \\\hline
 0     &&& a_1 & G \\[-2mm]
 a_1     &&&  &  \\[-2mm]
  & \Ddots    &&& & \\[-2mm]
  && a_1 && && a_n & G \\
\CodeAfter
\tikz\draw[line width=0.3mm, shorten >=2.5mm, shorten <=2.5mm, loosely dotted] (m2-7-4)--(m2-10-7);
\tikz\draw[line width=0.3mm, shorten >=2.5mm, shorten <=2.5mm, loosely dotted] (m2-7-5)--(m2-10-8);
        \tikz \draw[color=blue,thick,decorate,decoration={brace,mirror}] ([shift={(-1em,-0.2em)}]4-|1) -- ([shift={(-1em,0em)}]6.5-|1);
        \tikz \draw[color=blue,thick,decorate,decoration={brace,mirror}] ([shift={(-1em,0.2em)}]8-|1) -- ([shift={(-1em,0em)}]10.5-|1);
        \tikz \draw[color=blue,thick,->,out=230,in=130] ([shift={(-1.5em,-0.3em)}]5-|1) to node[left]{$1$} ([shift={(-1.5em,-0.1em)}]9-|1);
\end{pNiceArray}.
\]
Finally, we make a basis change to the rows of the first differential,
which happens to not change the second differential, resulting in:
\[
\begin{pNiceArray}{*3{@{}c}|c*4{@{}c}}[margin,name=m3]
\arrayrulecolor{darkred}
 G &&& \phantom{0} & \phantom{-\tfrac{a_2}{a_1}G} &  \\[-2mm]
   & \Ddots &&&&  &  \\[-2mm]
  &  & G &&&& \phantom{-\tfrac{a_n}{a_1}G} & \phantom{0} \\\hline
 -a_1 &&& \\[-2mm]
   & \Ddots && \\[-2mm]
 &  & -a_1 &  &&& \\\hline
      &&& a_1 & G \\[-2mm]
  &     &&& & \\[-2mm]
  && && && a_n & G \\
\CodeAfter
\tikz\draw[line width=0.3mm, shorten >=2.5mm, shorten <=2.5mm, loosely dotted] (m3-7-4)--(m3-9-7);
\tikz\draw[line width=0.3mm, shorten >=2.5mm, shorten <=2.5mm, loosely dotted] (m3-7-5)--(m3-9-8);
\end{pNiceArray},
\ 
\begin{pNiceArray}{*3{@{}c}|c*2{@{}c}|c*3{@{}c}}[margin]
\arrayrulecolor{darkred}
 -a_1 &&& -G & && \phantom{a_2} & \phantom{G} \\[-2mm]
 & \Ddots &&& \Ddots &&&  &  \\[-2mm]
 && -a_1 &&& -G &&& \phantom{a_n} & \phantom{G} \\
\end{pNiceArray}.
\]
This complex is isomorphic to a sum of $\stair_{(a_1, \dots, a_n)}$
and several grading shifted copies of $\stair_{(a_1)} \otimes \stair_{(a_1)}$.
By \cref{lem:summand}(ii), this implies that this complex is $Z$\nbd equivalent to $\stair_{(a_1, \dots, a_n)}$.
\end{proof}
For the proof of the next two lemmas, we will require \emph{cancellation}:
\begin{lemma}[\cite{zbMATH05155295}]\label{lem:cancellation}
A graded chain complex of the form
\[
\begin{tikzcd}[ampersand replacement=\&,column sep=2.8em]
\cdots \ar[r] \& R^a \ar{r}{%
\left(\begin{array}{c}
A \\\hline
B
\end{array}\right)} \& R\oplus R^b \ar{r}{%
\left(\begin{array}{@{}c@{\,}|@{\,}c@{}}
1 & C \\\hline
D & E
\end{array}\right)}
\& R\oplus R^c  \ar{r}{%
\left(\begin{array}{@{}c@{\,}|@{\,}c@{}}
F & G
\end{array}\right)} \& R^d \ar[r] \& \cdots \\
\hspace{-2.1em}\makebox[0pt][l]{over a graded ring $R$ is homogeneously homotopy equivalent to} \\
\cdots \ar[r] \& R^a \ar{r}{\bigl(B\bigr)} \& R^{b} \ar{r}{%
\bigl(E - DC\bigr)}
\& R^{c} \ar{r}{\bigl(G\bigr)} \& R^d \ar[r] \& \cdots.\qed
\end{tikzcd}
\]
\end{lemma}

\begin{lemma}\label{lem:staircase2}
If $a, b\geq 0$ are coprime, then $\stair_{(a)} \otimes \stair_{(b)} \simeq \stair_{(ab)}$.
\end{lemma}
\begin{proof}
Denote the bases of $\stair_{(a)}$ and $\stair_{(b)}$ by $x_1, x_2, x_3$ and $y_1, y_2, y_3$, respectively,
so that $d_{(a)}(x_1) = ax_3, d_{(a)}(x_2) = Gx_3, d_{(b)}(y_1) = by_3, d_{(b)}(y_2) = Gy_3$.
Then the following is a basis for $\stair_{(a)} \otimes \stair_{(b)}$:
\[
\begin{pmatrix}
x_1 \otimes y_1, \\
x_1 \otimes y_2, \\
x_2 \otimes y_1, \\
x_2 \otimes y_2\phantom{,}
\end{pmatrix}
\xrightarrow{d_0}
\begin{pmatrix}
x_1 \otimes y_3, \\
x_3 \otimes y_1, \\
x_2 \otimes y_3, \\
x_3 \otimes y_2\phantom{,}
\end{pmatrix}
\xrightarrow{d_1}
(x_3\otimes y_3).
\]
With respect to this basis, the differentials of $\stair_{(a)} \otimes \stair_{(b)}$ are
\[
d_0 = \begin{pmatrix}
b & G &   &   \\
a &   & G &   \\
  &   & b & G \\
  & a &   & G 
\end{pmatrix},\quad
d_1 = 
\begin{pmatrix}
a & -b & G & -G
\end{pmatrix}.
\]
We will now make a sequence of basis changes and cancellations (see \cref{lem:cancellation} for the latter).
Since $a,b$ are coprime, we may choose $\alpha,\beta\in\Z$ such that the matrix $\begin{psmallmatrix} b & \alpha \\ a & \beta \end{psmallmatrix}$ has determinant $1$, and thus inverse 
$\begin{psmallmatrix} \beta & -\alpha \\ -a & b \end{psmallmatrix}$.
For the $4\times 4$--matrix $T$ equal to the block sum $\begin{psmallmatrix} b & \alpha \\ a & \beta \end{psmallmatrix} \oplus \begin{psmallmatrix} 1 & 0 \\ 0 & 1 \end{psmallmatrix}$,
we find
\[
T^{-1}d_0 = 
\begin{pmatrix}
1 & \beta G & -\alpha G  &   \\
  & -a G    & b G        &   \\
  &         & b          & G \\
  & a       &            & G 
\end{pmatrix},\quad
d_1T = 
\begin{pmatrix}
0 & -1 & G & -G
\end{pmatrix}.
\]
Canceling twice yields a homogeneous homotopy equivalence between $\stair_{(a)} \otimes \stair_{(b)}$ and
the complex $\Z[G]^3 \longrightarrow \Z[G]^2$ with grading $t^0(q^2+q^2+q^0) + t^1(q^2+q^2)$
and differential
\[
d_0' =
\begin{pmatrix}
     & b   & G \\
 a   &     & G 
\end{pmatrix}.
\]
Changing the bases of both chain modules yields
\[
\begin{pmatrix}
-1 & 1 \\
0 & 1 
\end{pmatrix}
\cdot
d_0'
\cdot
\begin{pmatrix} b & \alpha & \\ a & \beta & \\ && 1 \end{pmatrix}
=
\begin{pmatrix}
 0  & -1      & 0 \\
 ab & a\alpha & G
\end{pmatrix}.
\]
Another cancellation gives the desired homogeneous homotopy equivalence to $\stair_{(ab)}$.
\end{proof}
\begin{lemma}\label{lem:staircase3}
For all $A = (a_1, \ldots, a_n), B = (b_1, \ldots, b_m)$ with $a_i \geq 0$, $a_i\neq 1$, $a_i|a_{i+1}$ and $b_i \geq 0$, $b_i\neq 1$, $b_i|b_{i+1}$ and $A\neq B$ we have that $\stair_A \not\sim \stair_B$.
\end{lemma}
\begin{proof}
For simplicity, we would like to have $n = m$. So, if $n\neq m$, we insert $|n-m|$ many entries equal to 1 at the beginning of $A$ or $B$. This modification preserves the $Z$\nbd equivalence class of $\stair_A$ and $\stair_B$, and also preserves the condition $A \neq B$. So we now assume w.l.o.g.~that $n = m$, and drop the assumption $a_i\neq 1, b_i\neq 1$.

Let $k \in \{1,\ldots,n\}$ such that $a_i = b_i$ for $i < k$ and $a_k \neq b_k$.
We assume w.l.o.g.~that $a_k$ does not divide $b_k$ (this may be achieved by switching $A$ and~$B$ if necessary).
Let $f\colon \stair_A\to \stair_B$ be a homogeneous chain map of quantum degree $0$.
Our goal is now to show that $f$ is not a $Z$\nbd isomorphism.
This will imply $\stair_A\not\sim \stair_B$ as desired.
Note that the chain
\[
\alpha = \Bigl(G^n, -a_1 G^{n-1}, a_1a_2 G^{n-2}, \dots, (-1)^n \prod_{i=1}^n a_i\Bigr)
\]
in the 0\nbd th chain module of $\stair_A$ is, up to sign, the unique homogeneous cycle of
quantum degree~$0$ that represents a generator of~$H_0(\stair_A\otimes Z)\cong Z$.
The analogous statement is true for the chain
\[
\beta = \Bigl(G^n, -b_1 G^{n-1}, b_1b_2 G^{n-2}, \dots, (-1)^n \prod_{i=1}^n b_i\Bigr)
\]
in the 0\nbd th chain module of $\stair_B$.
Hence $f_*\colon H_0(\stair_A\otimes Z) \to H_0(\stair_B\otimes Z)$
is an isomorphism
if and only if $f(\alpha) = \pm\beta$. Let us now show that $f(\alpha) \neq \pm\beta$.
Being homogeneous of degree~$0$,
the chain map $f$ in homological degree~0 is represented by a square matrix of size $(n+1)$ of the form
\[
\begin{pmatrix}
M_{11} & G M_{12} & G^2 M_{13} & \cdots \\
0      & M_{22}   & G M_{23}   & \cdots \\
0      & 0        & M_{33}     & \cdots \\
\vdots & \vdots   & \vdots     & \ddots
\end{pmatrix}
\]
with $M_{ij}\in\Z$. So the $(k+1)$\nbd st entry of the vector $f(\alpha)$ is divisible by $a_1 \cdots a_k$.
Since $a_1 \cdots a_k$ does not divide $b_1 \cdots b_k$, it follows that $f(\alpha)$ is not equal to~$\pm \beta$.
That concludes the proof.
\end{proof}
Now that we have seen some non-trivial examples of $Z$\nbd equivalence classes,
one is naturally led to the following questions.
\begin{question}\label{q:isoG}
What is the isomorphism type of $\mathcal{G}$?
\end{question}
\begin{question}\label{q:geography}
What is the geography of $S$, i.e.~which knot-like complexes $C\in\mathcal{G}$
do indeed occur as the $S$--invariant of a knot?
\end{question}
Regarding \cref{q:isoG}, \cref{prop:staircases} gives us the $\Z^{\oplus\infty}$ subgroup $\calstair \subset \mathcal{G}$. Actually, $\tfrac12(s_0, s_2-s_0, s_3 - s_0, \ldots)$ gives a surjection $u\colon \mathcal{G} \to \Z^{\oplus\infty}$, whose kernel contains a $\Z^{\oplus\infty}$ subgroup with basis
$\stair_{(p^{r+1})} \otimes \stair^*_{(p^r)}$, where $p$ ranges over all primes and $r$ ranges over the positive integers.
Regarding \cref{q:geography}, one may conjecture that the restriction of $u$ to the subgroup of $\mathcal{G}$ realized by knots is still surjective~\cite{zbMATH06296598,postcard,schuetz,LZ22}, and that all the classes in $\calstair \subset \mathcal{G}$ are indeed realized by knots.
However, as the following example suggests, there may very well be
knots $K$ with $S(K) \not \in \calstair$.
\begin{example}\label{ex:18}
The knot $K = 18^{nh}_{9772775}$ (see the tables~\cite{zbMATH07760154}),
which is one of the knots discovered by Sch\"utz with $s_c(K) = s_0(K)$ for
all primes~$c$ but $s^{\Z} \neq s_0(K)$, may be computed using \cite{khoca,homca} to have $S(K)$ represented by the following complex:
\[
%
C\ \coloneqq
\begin{tikzcd}
 & q^4\Z[G] \ar[rd,"2"] \\
 & q^4\Z[G] \ar[r,"4"] & q^4\Z[G] \\
q^2\Z[G] \ar[ru,"-G"]\ar[r,"4"]\ar[rd,"2",darkgreen]
 & q^2\Z[G] \ar[ru,"G"]\\
\color{darkgreen}{q^0\Z[G]} \ar[r,"G",darkgreen] & \color{darkgreen}{q^2\Z[G]}
\end{tikzcd}
\]
\smallskip

Ignoring the part of the complex $C$ printed in green
yields the complex $C^1 = \stair_{(2)} \otimes \stair^*_{(4)}$ discussed in 
\cref{ex:staircasetimesdual} below, shifted by $q^4$,
which lies in $\calstair$.
On the other hand,
it seems plausible that the class of $C$ does not lie in~$\calstair$.
\end{example}

\section{Further properties of \texorpdfstring{$S$}{S}} \label{sec:further}
Let us review the definition of Sch\"utz's invariant $s^{\Z}$~\cite{schuetz} in terms of filtrations.
For a diagram $D$ of a knot $K$ and an even integer $k$,
let the subcomplex $\mathcal{F}_k (C(D) \otimes Z) \subset C(D) \otimes Z$ be defined as
\[
\{ x\otimes 1 \mid x\in C(D) \text{ is homogeneous of quantum grading } k\}.
\]
This is a \emph{filtration} on $C(D) \otimes Z$, i.e.
\[
\cdots \subset \mathcal{F}_k (C(D) \otimes Z) \subset \mathcal{F}_{k-2} (C(D) \otimes Z) \subset \cdots
\]
Such a filtration induces a filtration on homology,
namely by setting
\[
\mathcal{F}_k H(C(D) \otimes Z) \coloneqq
\{ [x] \in H(C(D) \otimes Z) \mid x \in \mathcal{F}_k (C(D) \otimes Z)\}.
\]
When identifying $H_0(C(D) \otimes Z)$ with $\Z$,
an example of such a filtration is 
\[
\newcommand{\veq}{\rotatebox{90}{=}}
\begin{array}{*{15}{@{\ }c}}
\cdots & \subset & 0 & \subset & 0 & \subset & 6\Z & \subset & 2\Z & \subset & \Z & \subset & \Z & \subset & \cdots \\
&& \veq && \veq&& \veq&& \veq&& \veq&& \veq \\
\cdots & \subset & \mathcal{F}_8 & \subset & \mathcal{F}_6 & \subset & \mathcal{F}_4 & \subset & \mathcal{F}_2 & \subset & \mathcal{F}_0 & \subset & \mathcal{F}_{-2} & \subset & \cdots
\end{array}
\]
The filtration on $H_0(C(D) \otimes Z)$
can be shown to stabilize as $0$ on the left and $H_0(C(D) \otimes Z)$ on the right,
and so it is completely encoded by the tuple $(k_0, k_1, \ldots, k_n) \in (2\Z) \times (\Z_{\geq 1})^n$
satisfying

\begin{align*}
\mathcal{F}_{k_0 + 2} & = 0, & \mathcal{F}_{k_0 - 2}/\mathcal{F}_{k_0} & \cong k_1\Z, \\
\mathcal{F}_{k_0} & \neq 0, & \mathcal{F}_{k_0 - 4}/\mathcal{F}_{k_0 - 2} & \cong k_2\Z, \\
\mathcal{F}_{k_0 - 2n + 2} & \neq H_0(C(D) \otimes Z), & & \vdots \\
\mathcal{F}_{k_0 - 2n} & = H_0(C(D) \otimes Z),  & \mathcal{F}_{k_0 - 2n}/\mathcal{F}_{k_0 - 2n + 2} & \cong k_n \Z.
\end{align*}
The invariant $s^{\Z}(K)$ is defined precisely as this tuple.
Its first entry can be shown to equal $s_0(K)$. Its length minus one (i.e.~$n$) is denoted by $\gl(K)$.
For the above filtration,
one would e.g.~find $s^{\Z}(K) = (4,3,2)$ and $\gl(K) = 2$.
Let us now prove that for all knot $K$ and $K'$, $s^{\Z}(K) = s^{\Z}(K')$ follows from $S(K) = S(K')$, in other words the\ldots
\begin{proof}[\ldots second half of the proof of \cref{thm:stronger}]
Let $D$ and $D'$ be diagrams of $K$ and $K'$, respectively.
Since $S(K) = S(K')$, there exist $Z$\nbd isomorphisms $f\colon C(D)\to C(D')$, $g\colon C(D')\to C(D)$.
Since $f$ and $g$ are homogeneous, the maps $f_*$ and $g_*$ are \emph{filtered} maps, i.e. $f_*\mathcal{F}_k H(C(D)\otimes Z) \subset \mathcal{F}_k H(C(D')\otimes Z)$ and similarly for $g_*$.
This implies that the filtrations on $H(C(D)\otimes Z)$ and $H(C(D')\otimes Z$ are isomorphic.
Since $s^{\Z}$ is defined in terms of those filtrations, $s^{\Z}(K) = s^{\Z}(K')$ follows.
\end{proof}
\begin{remark}
In practice, one computes $s^{\Z}(K)$ from a given chain complex~$C$ representing $S(K)$
by determining the filtration on $H_0(C\otimes Z)$; either directly, or via the associated spectral sequence. For example, let $(a_1, \ldots, a_n)$ be a tuple as in \cref{def:staircase},
and assume for simplicity that $a_i\neq 0$.
Then the reader may convince themself that
\[
S(K) = [\stair_{(a_1, \ldots, a_n)}] \quad\Rightarrow\quad s^{\Z}(K) = (0),
\]
and (more interestingly!) that
\[
S(K) = [\stair^*_{(a_1, \ldots, a_n)}] \quad\Rightarrow\quad s^{\Z}(K) = (0, a_n, a_{n-1}, \ldots, a_1).
\]
This offers an alternative proof of \cref{lem:staircase3}.
\end{remark}
Let us now exhibit examples of knots that may be distinguished by $S$,
but not by $s_c$ or $s^{\Z}$.
\begin{example}\label{ex:staircasetimesdual}
A quick calculation (similar to \cref{lem:staircase1,lem:staircase2})
shows that for $k \geq 1$, the complex $\stair_{(2^k)} \otimes \stair^*_{(2^{k+1})}$
is isomorphic to the
sum of the two complexes (both supported in homological degrees $-1$, $0$ and $1$)
\[
\hspace{28pt}
\begin{tikzcd}
 & q^0\Z[G] \ar[r,"2^k"] & q^0\Z[G] \\
q^{-2}\Z[G] \ar[ru,"-G"] \ar[r,"2^k"] & q^{-2}\Z[G]\ar[ru,"G"]
\end{tikzcd}
\]
and\smallskip

\[
C^k\ \coloneqq
\begin{tikzcd}
 & q^0\Z[G] \ar[rd,"2^k"] \\
 & q^0\Z[G] \ar[r,"2^{k+1}"] & q^0\Z[G] \\
q^{-2}\Z[G] \ar[ru,"-G"]\ar[r,"2^{k+1}"] & q^{-2}\Z[G] \ar[ru,"G "]
\end{tikzcd}
\]
Since $C^k$ is knot-like, 
we have $\stair_{2^k} \otimes \stair^*_{2^{k+1}} \sim C^k$ by \cref{lem:summand}.
In particular, if \cref{conj:whitehead} holds, then $C^k$ represents
$S$ of the knot
\[
W^+_{3k}(T_{2,2k+1}) \#\, {-W^+_{3k+3}(T_{2,2k+3})}.
\]
One easily calculates that $s_c(C^k) = 0$ for all characteristics~$c$ and all $k$.
Moreover one finds $s^{\Z}(C^k) = (0,2)$ and $s^{\Z}((C^k)^*) = (0)$ for all~$k$.
However, it follows from \cref{prop:staircases} that $C^k \not\sim C^{\ell}$ for $k \neq \ell$.
\end{example}

Just like the original Rasmussen invariant, $S$ is not just a concordance homomorphism,
but also yields a lower bound for the smooth slice genus, which we explore now.
\begin{definition}
For two knot-like complexes $C$ and $C'$,
denote by 
$d(C, C')$ 
the minimal $n\geq 0$ such that there exists $Z$\nbd isomorphisms
$f\colon C\to C'$ and $g\colon C'\to C$ of quantum degree $-2n$.
\end{definition}
The existence of such an integer $n$ is a consequence of the following.
\begin{proposition}\label{prop:slicegenus}
The function $d$ induces a metric on $\mathcal{G}$, which we also denote by~$d$.
This metric is a lower bound for the smooth oriented cobordism distance,
i.e.\ $d(S(K), S(J)) \leq g_4(K\#\, {-J})$.
\end{proposition}
\begin{proof}
On the set of knot-like complexes, one easily checks that $d$ is a pseudo-metric.
Since $C$ and $C'$ are by definition $Z$\nbd equivalent if and only if $d(C, C') = 0$,
$d$ descends to a metric on $\mathcal{G}$.
The lower bound immediately follow from property \eqref{axiom:cobo}.
\end{proof}
This bound subsumes the bounds
\begin{align*}
\frac{|s_c(K) - s_c(J)|}{2} & \leq g_4(K\#\, {-J}), \\
\left|\frac{s_0(K) - s_0(J)}{2}-\gl(K)+\gl(J)\right| & \leq g_4(K\#\, {-J})
\end{align*}
coming from the Rasmussen invariants and Sch\"utz's invariant
(see \cref{prop:rasmussenF}(iii) for the first bound
and \cite{schuetz} as well as \cite[Proof of Lemma 3.12]{sqz} for the second bound):
\begin{proposition}
For all knots $K$ and $J$ and all $c$ equal to zero or a prime,
\[
\frac{|s_c(K) - s_c(J)|}{2} \leq d(S(K), S(J)).
\]
Moreover
\[
\left|\frac{s_0(K) - s_0(J)}{2}-\gl(K)+\gl(J)\right| \leq d(S(K), S(J)).
\]
\end{proposition}
\begin{proof}
The first inequality may be proven very similarly to the first half of \cref{thm:stronger}.
Let us now show the second inequality. Fix diagrams $D_K$ and $D_J$ of $K$ and $J$, respectively.
Let $f\colon C(D_K) \to C(D_J)$ be a $Z$\nbd isomorphism of quantum degree $-2d(S(K), S(J))$.
Note that the maximal even integer $k$ such that $\mathcal{F}_k H(C(D_K)\otimes Z) = H(C(D_K)\otimes Z)$ is $k = s_0(K) - 2\gl(K)$ (and similarly for $J$).
Much as in the proof of the second half of \cref{thm:stronger}, we now use that $f_*$ is a filtered map of degree $-2d(S(K), S(J))$, and so
\[
f_*( \mathcal{F}_{s_0(K) - 2\gl(K)} H(C(D_K)\otimes Z))
\]
is contained in
\[
\mathcal{F}_{s_0(K) - 2\gl(K) - 2d(S(K), S(J))} H(C(D_J)\otimes Z).
\]
Since $f_*$ is surjective and $\mathcal{F}_{s_0(K) - 2\gl(K)} H(C(D_K)\otimes Z) = H(C(D_K)\otimes Z)$,
it follows that $\mathcal{F}_{s_0(K) - 2\gl(K) - 2d(S(K), S(J))} H(C(D_J)\otimes Z) = H(C(D_J)\otimes Z)$.
Hence we have
\begin{align*}
s_0(K) - 2\gl(K)-2d(S(K), S(J)) & \ \leq\  s_0(J) - 2\gl(J) \ \Rightarrow\\
\frac{ s_0(K) - s_0(J)}2 - \gl(K) + \gl(J)& \ \leq\  d(S(K), S(J)). 
\end{align*}
Switching the roles of $K$ and $J$ concludes the proof.
\end{proof}
A knot $K$ is called \emph{squeezed} if there is a smooth genus-minimizing cobordism
from a negative to a positive knot factorizing through $K$~\cite{sqz}.
Many classes of knots, such as alternating and quasipositive knots, are squeezed.
On squeezed knots, many concordance invariants (such as slice-torus invariants)
contain no more information than the Rasmussen invariant.
This is also true for $S$.
\begin{proposition}\label{prop:sqz}
If $K$ is a squeezed knot (e.g.~if $K$ is quasipositive),
or if $K$ has thin reduced Khovanov homology (e.g.~if $K$ is quasi-alternating),
then $S(K)$ is represented by $t^0 q^{s_0(K)} \Z[G]$.
\end{proposition}
For the proof,
we need the following property of Khovanov homology,
which requires a quick dip into the actual definition of $C(D)$.
\begin{lemma}\label{lem:khtorus}
If $D$ is a positive diagram of a knot $K$, then $C_0(D)$ has a $\Z[G]$\nbd basis
consisting of homogeneous elements with quantum degree at least~$2g_4(K)$.
\end{lemma}
\begin{proof}
Recall that $C(D)$ is defined via the cube of resolutions (see e.g.~\cite{MR1917056}),
which, since $D$ is positive, has a single vertex in homological degree zero.
Elements of $C_0(D)$ are thus given by a tensor product over $\Z[G]$,
with tensor factor $\Z[G]$ for the Seifert circle containing the base point
(since we consider reduced homology, see e.g.~\cite{ilm21}),
and a tensor factor $A = q\Z[G][X]/(X^2+GX)$ for all other Seifert circles.
A basis for $C_0(D)$ is given by pure tensors with factor $1\in\Z[G]$ for  the Seifert circle containing the base point, and factors $1$ or $X\in A$ for all other Seifert circles. The basis element with minimal quantum degree is
$1\otimes X \otimes \dots \otimes X$, which has degree $1 - k$, where $k$
is the number of Seifert circles, or rather $1 + c - k$,
taking the global degree shift by~$q^c$ into account,
where $c$ is the number of crossings of~$D$.
Because $1 + c - k = 2g_4(K)$~\cite{rudolph_QPasObstruction},
the desired statement follows.
\end{proof}
\begin{proof}[Proof of \cref{prop:sqz}]
For thin knots $K$, it has been shown in \cite[Lemma~3.32]{ilm21}
that $C(K)$ has a rank one summand, which implies the statement.
Let us now turn to squeezed knots.
Consider the function $\Z\to\mathcal{G}$ sending $n$ to the class of the
rank one complex $t^0 q^{2n}\Z[G]$. This is clearly an isometric embedding with respect
to the standard metric on $\Z$ and the metric $d$ on $\mathcal{G}$ defined above.
The `squeezing framework' Lemma~3.8 in \cite{sqz} now implies the statement of the
proposition, provided that the following three conditions hold:
\begin{enumerate}[label=(\roman*)]
\item $d(S(K), S(J)) \leq g_4(K\#\, {-J})$ for all knots $K, J$.
\item For all positive torus knots $T_+$ and negative torus knots $T_-$,
the class $S(T_{\pm})$ is represented by $t^0 q^{\pm 2g_4(T_{\pm})}\Z[G]$.
\item If $d(t^0 q^{2m}, C) + d(C, t^0 q^{2n}) = m - n$ holds for some integers $m \geq n$ and
a knot-like complex $C$, then $C \sim t^0 q^{2k}\Z[G]$ for some $k\in \{n, \ldots, m\}$.
\end{enumerate}

\textbf{(i):} This has been shown in \cref{prop:slicegenus}.

\textbf{(ii):}
Let $T_+$ be a positive torus knot $T_+$. Let us
fix a positive diagram $D$ of $T_+$ and
construct $Z$\nbd isomorphisms
$f\colon C(D) \to t^0 q^{2g_4(T_+)}\Z[G]$
and $g\colon t^0 q^{2g_4(T_+)}\Z[G] \to C(D)$
of quantum degree 0.
Since there is a cobordism of genus $g_4(T_+)$ from $T_+$ to the unknot,
by \eqref{axiom:cobo}
there exists a $Z$\nbd isomorphism $C(D) \to t^0q^0\Z[G]$
of quantum degree $-2g_4(T_+)$.
Shifting the degree of the target module of that map by $q^{2g_4(T_+)}$
yields the desired map $f$.
To construct~$g$, start with 
a $Z$\nbd isomorphism $g'\colon t^0 q^{2g_4(T_+)}\Z[G] \to C(D)$
of any degree $2r$.
If $r \geq 0$, simply let $g = G^r g'$.
Let us now consider the case $r < 0$.
By \cref{lem:khtorus},
the chain module $C_0(D)$
has a $\Z[G]$\nbd basis $b_1, \ldots, b_n$ such that $\min_i r_i = g_4(T_+)$
for $r_i = \tfrac12\qdeg b_i$. With respect to that basis, we have
\[
g'(1) = (\lambda_1 G^{r_1-r-g_4(T_+)}, \dots, \lambda_n G^{r_n-r-g_4(T_+)})
\]
for some $\lambda_i \in \Z$.
It follows that $g'(1)$ is divisible by $G^{-r}$.
Now $g = G^r g'$ is a $Z$\nbd isomorphism of the desired degree.
This concludes the proof of condition (ii) for $T_+$.
For negative torus knots $T_-$, condition (ii) follows by mirroring.

\textbf{(iii):}
By assumption, there exist $Z$\nbd isomorphisms $f\colon t^0 q^{2m}\Z[G] \to C$
and $g\colon C \to t^0 q^{2n}\Z[G]$ 
such the sum of $a = \qdeg f$ and $b = \qdeg g$ equals $2(n - m)$. Shifting degrees gives us 
the following homogeneous chain maps of quantum degree zero:
$f'\colon t^0q^{2m}\Z[G] \to q^{-a}C$ and 
$g'\colon q^{-a}C \to t^0q^{2m}\Z[G]$.
Since $f', g'$ still are $Z$\nbd isomorphisms,
it follows that $d(t^0q^{2m}\Z[G], q^{-a}C) = 0$, so $t^0q^{2m}\sim q^{-a}C$.
Since $C$ is knot-like, $a$ must be even. The statement follows.
\end{proof}
The triviality of $S$ on squeezed knots
may be used to prove the last part of \cref{thm:kernel}.
\begin{proof}[Proof of {\cref{thm:kernel}(iii)}]
For each $n\geq 1$, pick a squeezed knot~$J'_n$ that is topologically concordant to $-W^+_{3n}(T_{2,2n+1})$.
Such knots $J'_n$ exist because every topological concordance class contains strongly quasipositive knots~\cite{borodzikfeller}, which are squeezed.
Moreover, pick a topologically slice squeezed knot $J''_n$ such that $s_0(J''_n) = -s_0(J'_n)$. For example, $J''_n$ can be taken as a connected sum of copies of the $P(3,-5,-7)$ pretzel knot or its mirror image, since this pretzel knot is strongly quasipositive (thus squeezed), has Alexander polynomial~1 (and is thus topologically slice), and satisfies $s_0(P(3,-5,-7)) = 2$.
Let $J_n = J'_n \# J''_n$. It follows that $K_n = W^+_{3n}(T_{2,2n+1}) \# J_n$ is topologically slice. Moreover, $J_n$ is squeezed and $s_0(J_n) = 0$, and thus by \cref{prop:sqz}, $S(J_n)$ is trivial.
\end{proof}

\begin{remark}
By the same argument as in the above proof of {\cref{thm:kernel}(iii)},
it follows that $K_n$ is a basis of a $\Z^{\infty}$ subgroup
of the quotient $\mathcal{C} / \mathcal{Z}$, where $\mathcal{Z}$ is the subgroup of squeezed knots.
\end{remark}

\bibliographystyle{myamsalpha}
\bibliography{References}
\end{document}